\documentclass[11pt, a4paper, twoside,leqno]{amsart}
\usepackage[centering, totalwidth = 420pt, totalheight = 620pt]{geometry}
\usepackage{amssymb, amsmath, amsthm}
\usepackage{enumerate, microtype, stmaryrd, url,mathrsfs} 
\usepackage[latin1]{inputenc}
\usepackage{color}
\definecolor{darkgreen}{rgb}{0,0.45,0}
\usepackage[colorlinks,citecolor=darkgreen,linkcolor=darkgreen]{hyperref}
\usepackage[british]{babel}

\DeclareMathAlphabet{\mathbf}{OT1}{cmr}{b}{n}

\usepackage[arrow, matrix, tips, curve, graph, rotate]{xy}
\SelectTips{cm}{10}

\makeatletter
\def\matrixobject@{%
  \edef \next@{={\DirectionfromtheDirection@ }}%
  \expandafter \toks@ \next@ \plainxy@
  \let\xy@@ix@=\xyq@@toksix@
  \xyFN@ \OBJECT@}
\let\xy@entry@@norm=\entry@@norm
\def\entry@@norm@patched{%
  \let\object@=\matrixobject@
  \xy@entry@@norm }
\AtBeginDocument{\let\entry@@norm\entry@@norm@patched}
\makeatother

\newcommand{\twocong}[2][0.5]{\ar@{}[#2] \save ?(#1)*{\cong}\restore}
\newcommand{\twoeq}[2][0.5]{\ar@{}[#2] \save ?(#1)*{=}\restore}
\newcommand{\rtwocell}[3][0.5]{\ar@{}[#2] \ar@{=>}?(#1)+/l 0.2cm/;?(#1)+/r 0.2cm/^{#3}}
\newcommand{\ltwocell}[3][0.5]{\ar@{}[#2] \ar@{=>}?(#1)+/r 0.2cm/;?(#1)+/l 0.2cm/^{#3}}
\newcommand{\ltwocello}[3][0.5]{\ar@{}[#2] \ar@{=>}?(#1)+/r 0.2cm/;?(#1)+/l 0.2cm/_{#3}}
\newcommand{\dtwocell}[3][0.5]{\ar@{}[#2] \ar@{=>}?(#1)+/u  0.2cm/;?(#1)+/d 0.2cm/^{#3}}
\newcommand{\dltwocell}[3][0.5]{\ar@{}[#2] \ar@{=>}?(#1)+/ur  0.2cm/;?(#1)+/dl 0.2cm/^{#3}}
\newcommand{\drtwocell}[3][0.5]{\ar@{}[#2] \ar@{=>}?(#1)+/ul  0.2cm/;?(#1)+/dr 0.2cm/^{#3}}
\newcommand{\dthreecell}[3][0.5]{\ar@{}[#2] \ar@3{->}?(#1)+/u  0.2cm/;?(#1)+/d 0.2cm/^{#3}}
\newcommand{\utwocell}[3][0.5]{\ar@{}[#2] \ar@{=>}?(#1)+/d 0.2cm/;?(#1)+/u 0.2cm/_{#3}}
\newcommand{\dtwocelltarg}[3][0.5]{\ar@{}#2 \ar@{=>}?(#1)+/u  0.2cm/;?(#1)+/d 0.2cm/^{#3}}
\newcommand{\utwocelltarg}[3][0.5]{\ar@{}#2 \ar@{=>}?(#1)+/d  0.2cm/;?(#1)+/u 0.2cm/_{#3}}

\newdir{(}{{}*!<0em,-.14em>-\cir<.14em>{l^r}}
\newdir{ (}{{}*!/-5pt/\dir{(}}
\newdir{ >}{{}*!/-5pt/\dir{>}}

\newcommand{\cat}[1]{\mathbf{#1}}

\newcommand{\id}{\mathrm{id}}
\newcommand{\thg}{{\mathord{\text{--}}}}

\newcommand{\spn}[1]{{\langle{#1}\rangle}}

\newcommand{\defeq}{\mathrel{\mathop:}=}
\newcommand{\cd}[2][]{\vcenter{\hbox{\xymatrix#1{#2}}}}

\renewcommand{\phi}{\varphi}

\newcommand{\C}{{\mathscr C}}

\newcommand{\xtor}[1]{\cdl[@1]{{} \ar[r]|-{\object@{|}}^{#1} & {}}}

\makeatletter

\def\hookleftarrowfill@{\arrowfill@\leftarrow\relbar{\relbar\joinrel\rhook}}
\def\twoheadleftarrowfill@{\arrowfill@\twoheadleftarrow\relbar\relbar}
\def\leftbararrowfill@{\arrowdoublefill@{\leftarrow\mkern-5mu}\relbar\mapstochar\relbar\relbar}
\def\Leftbararrowfill@{\arrowdoublefill@{\Leftarrow\mkern-2mu}\Relbar\Mapstochar\Relbar\Relbar}
\def\leftringarrowfill@{\arrowdoublefill@{\leftarrow\mkern-3mu}\relbar{\mkern-3mu\circ\mkern-2mu}\relbar\relbar}
\def\lefttriarrowfill@{\arrowfill@{\mathrel\triangleleft\mkern0.5mu\joinrel\relbar}\relbar\relbar}
\def\Lefttriarrowfill@{\arrowfill@{\mathrel\triangleleft\mkern1mu\joinrel\Relbar}\Relbar\Relbar}

\def\hookrightarrowfill@{\arrowfill@{\lhook\joinrel\relbar}\relbar\rightarrow}
\def\twoheadrightarrowfill@{\arrowfill@\relbar\relbar\twoheadrightarrow}
\def\rightbararrowfill@{\arrowdoublefill@{\relbar\mkern-0.5mu}\relbar\mapstochar\relbar\rightarrow}
\def\Rightbararrowfill@{\arrowdoublefill@{\Relbar\mkern-2mu}\Relbar\Mapstochar\Relbar\Rightarrow}
\def\rightringarrowfill@{\arrowdoublefill@\relbar\relbar{\mkern-2mu\circ\mkern-3mu}\relbar{\mkern-3mu\rightarrow}}
\def\righttriarrowfill@{\arrowfill@\relbar\relbar{\relbar\joinrel\mkern0.5mu\mathrel\triangleright}}
\def\Righttriarrowfill@{\arrowfill@\Relbar\Relbar{\Relbar\joinrel\mkern1mu\mathrel\triangleright}}

\def\leftrightarrowfill@{\arrowfill@\leftarrow\relbar\rightarrow}
\def\mapstofill@{\arrowfill@{\mapstochar\relbar}\relbar\rightarrow}

\renewcommand*\xleftarrow[2][]{\ext@arrow 20{20}0\leftarrowfill@{#1}{#2}}
\providecommand*\xLeftarrow[2][]{\ext@arrow 60{22}0{\Leftarrowfill@}{#1}{#2}}
\providecommand*\xhookleftarrow[2][]{\ext@arrow 10{20}0\hookleftarrowfill@{#1}{#2}}
\providecommand*\xtwoheadleftarrow[2][]{\ext@arrow 60{20}0\twoheadleftarrowfill@{#1}{#2}}
\providecommand*\xleftbararrow[2][]{\ext@arrow 10{22}0\leftbararrowfill@{#1}{#2}}
\providecommand*\xLeftbararrow[2][]{\ext@arrow 50{24}0\Leftbararrowfill@{#1}{#2}}
\providecommand*\xleftringarrow[2][]{\ext@arrow 10{26}0\leftringarrowfill@{#1}{#2}}
\providecommand*\xlefttriarrow[2][]{\ext@arrow 80{24}0\lefttriarrowfill@{#1}{#2}}
\providecommand*\xLefttriarrow[2][]{\ext@arrow 80{24}0\Lefttriarrowfill@{#1}{#2}}

\renewcommand*\xrightarrow[2][]{\ext@arrow 01{20}0\rightarrowfill@{#1}{#2}}
\providecommand*\xRightarrow[2][]{\ext@arrow 04{22}0{\Rightarrowfill@}{#1}{#2}}
\providecommand*\xhookrightarrow[2][]{\ext@arrow 00{20}0\hookrightarrowfill@{#1}{#2}}
\providecommand*\xtwoheadrightarrow[2][]{\ext@arrow 03{20}0\twoheadrightarrowfill@{#1}{#2}}
\providecommand*\xrightbararrow[2][]{\ext@arrow 01{22}0\rightbararrowfill@{#1}{#2}}
\providecommand*\xRightbararrow[2][]{\ext@arrow 04{24}0\Rightbararrowfill@{#1}{#2}}
\providecommand*\xrightringarrow[2][]{\ext@arrow 01{26}0\rightringarrowfill@{#1}{#2}}
\providecommand*\xrighttriarrow[2][]{\ext@arrow 07{24}0\righttriarrowfill@{#1}{#2}}
\providecommand*\xRighttriarrow[2][]{\ext@arrow 07{24}0\Righttriarrowfill@{#1}{#2}}

\providecommand*\xmapsto[2][]{\ext@arrow 01{20}0\mapstofill@{#1}{#2}}
\providecommand*\xleftrightarrow[2][]{\ext@arrow 10{22}0\leftrightarrowfill@{#1}{#2}}
\providecommand*\xLeftrightarrow[2][]{\ext@arrow 10{27}0{\Leftrightarrowfill@}{#1}{#2}}

\makeatother

\theoremstyle{plain}
\newtheorem{Thm}{Theorem}
\newtheorem{Prop}[Thm]{Proposition}

\newtheorem{Lemma}[Thm]{Lemma}

\theoremstyle{definition}
\newtheorem{Defn}[Thm]{Definition}

\newtheorem{Ex}[Thm]{Example}

\begin{document}
\leftmargini=2em 

\title{When coproducts are biproducts}

\author{Richard Garner} 
\address{Department of Mathematics, Macquarie University, NSW 2109,
  Australia} 
\email{richard.garner@mq.edu.au}

\author{Daniel Sch\"appi} 
\address{School of Mathematics and Statistics, University of
  Sheffield, Sheffield, S3 7RH, UK}
\email{d.schaeppi@sheffield.ac.uk} 

\thanks{The first author gratefully acknowledges the support of
Australian Research Council Discovery Projects DP110102360 and
DP130101969. The second author gratefully ackowledges the support of the
Swiss National Foundation Fellowship P2SKP2\_148479.}


\maketitle

\begin{abstract}
  Among monoidal categories with finite coproducts preserved by
  tensoring on the left, we characterise those with finite
  \emph{biproducts} as being precisely those in which the initial
  object and the coproduct of the unit with itself admit right duals.
  This generalises Houston's result that any compact closed category
  with finite coproducts admits biproducts.
\end{abstract}

\section{Background and statement of results}
\label{sec:background}

Recall that a monoidal category is \emph{compact closed} (also
\emph{autonomous}) when every object has both a left and right dual;
key examples include the categories of finite dimensional vector
spaces, and of sets and relations. In~\cite{Houston2008Finite},
Houston proves that in a compact closed category, finite products and
coproducts coincide; more precisely, they are
\emph{biproducts}:

\begin{Defn}
  \label{def:1}
  Let $\C$ be a category with a \emph{zero object}: an object
  $0 \in \C$ which is initial and terminal. A coproduct cocone
  $(\iota_i \colon A_i \rightarrow A)_{i \in I}$ in $\C$ is a
  \emph{biproduct}~\cite[\S VIII.2]{Mac-Lane1971Categories} if
  \begin{equation}
    \label{eq:3}
    (\pi_k \colon A \rightarrow A_k)_{k \in I}
  \end{equation}
  is a product cone, where $\pi_k$ is the unique morphism with
  $\pi_k \iota_k=1_{A_k}$ and with $\pi_k \iota_i = 0$ for $i \neq k$;
  here, for $X,Y \in \C$, $0 \colon X \rightarrow Y$ denotes the
  composite of unique maps $X \rightarrow 0 \rightarrow Y$.
\end{Defn}

Houston's proof does not adapt to give a characterisation of
categories with biproducts among the (different) class of symmetric
monoidal closed categories; in a question on
MathOverflow~\cite{Barton2010Semiadditivity}, Barton asked whether
such a characterisation could be given as the two requirements that
finite coproducts exist, and that the initial object and the coproduct
of the unit with itself have duals. After helpful conversations with
Mike Shulman, the second author was able to answer this question
affirmatively; some time later, the first author, inspired by
discussions with James Dolan, found a simpler version of the proof
which generalises to the non-symmetric monoidal case (and thus
recovers Houston's result). The goal of this note, then, is to give a
streamlined proof of:

\begin{Thm}
  \label{thm:1}
  If $\C = (\C, \otimes, I)$ is a monoidal category possessing finite
  coproducts preserved by each $A \otimes (\thg)$, then $\C$ has a
  zero object and finite biproducts if and only if the initial object
  $0$ and coproduct $I+I$ have right duals.
\end{Thm}

In fact, we prove something slightly more general. When $\C$ has finite
coproducts, the existence of biproducts is equivalent to
\emph{semi-additivity}: the existence of commutative monoid structures
on the hom-sets which are preserved by composition in each variable.

\begin{Thm}
  \label{thm:3}
  If $\C$ is a monoidal category which admits an initial object $0$
  and a coproduct $I+I$ of the unit object with itself, with both
  of these colimits being preserved by each $A \otimes (\thg)$, then $\C$ is
  semi-additive if and only if $0$ and $I+I$ have right duals.
\end{Thm}

A key ingredient in the proof of these theorems is the ``terminal
object lemma'':

\begin{Lemma}
  \label{lemma:1}
  An object $T$ of a category $\C$ is terminal if and only if there is
  a family of maps
  $(\varepsilon_C \colon C \rightarrow T)_{C \in \C}$, natural in $C$,
  for which $\varepsilon_T=1_T \colon T \rightarrow T$.
\end{Lemma}

\begin{proof}
  The ``only if'' follows as the unique morphisms $C \rightarrow T$
  are natural in $C$. Conversely, given $\varepsilon_{(\thg)}$, there
  is the map $\varepsilon_C \colon C \rightarrow T$ from each
  $C \in \C$; to show unicity, we use $\varepsilon_T=1_T$ and
  naturality of $\varepsilon$ to see that $f = \varepsilon_T f =
  \varepsilon_C$ for any
  $f \colon C \rightarrow T$.
\end{proof}

From this we recover the following well-known result, which provides
the link between Theorems~\ref{thm:1} and~\ref{thm:3}.

\begin{Prop}
  \label{prop:1}
  If $\C$ is semi-additive with an initial object $0$, then the
  initial object is a zero object and any finite coproduct that exists
  in $\C$ is a biproduct.
\end{Prop}

\begin{proof}
  The neutral elements of the monoids $\C(C,0)$ give a natural family
  $(C \rightarrow 0)_{C \in \C}$, so that $0$ is terminal by
  Lemma~\ref{lemma:1}. Suppose now that
  $(\iota_i \colon A_i \rightarrow A)_{i \in I}$ is a finite coproduct
  cocone. We will use Lemma~\ref{lemma:1} to show that the cone
  $\pi = (\pi_k \colon A \rightarrow A_k)_{k \in I}$ of \eqref{eq:3}
  is terminal among cones over the $A_k$'s. Given a cone
  $f = (f_k \colon B \rightarrow A_k)_{k \in I}$, we define
  ${\varepsilon_{f}=\Sigma_{i \in I} \iota_i f_i \colon B \rightarrow
  A}$.
  Then $\pi_j \varepsilon_{f}=\Sigma_i \pi_j \iota_i f_i=f_j$, so
  $\varepsilon_{f} \colon f \rightarrow \pi$ is a map of cones;
  moreover, $\varepsilon_{(\thg)}$ is natural as composition in $\C$
  is bilinear. Finally, we have
  $\varepsilon_{\pi} \iota_j= \textstyle\Sigma_i \iota_i \pi_i
  \iota_j=\iota_j$ and so $\varepsilon_{\pi}=1_A$.
\end{proof}

\section{Proofs and examples}
\label{sec:result}

Our main result is a consequence of the following necessary and
sufficient condition for a reasonable category $\C$ to be
semi-additive. We say that $\C$ has \emph{$2$-fold copowers} if all
coproducts $A+A$ exist, and that it has \emph{binary copowers} if, for
each $n \in \mathbb N$, all $2^n$-fold coproducts $A + \cdots + A$
exist; which is so just when $\C$ has $2$-fold copowers and an initial
object $0$. In a category with binary copowers, the coproducts $0+A$
and $A+0$ always exist (since they can be taken to be $A$), and so we
can talk about \emph{counital comagmas}: objects $A$ endowed with a
comultiplication $\delta \colon A \rightarrow A+A$ and a counit
$\varepsilon \colon A \rightarrow 0$ such that
$(\varepsilon + 1_A) \delta=1_A=(1_A+\varepsilon) \delta$.

\begin{Prop}\label{prop:2}
  Let $\C$ be a category with binary copowers. The following conditions are
  equivalent:
  \begin{enumerate}[(i)]
  \item The category $\C$ is semi-additive;
  \item The initial object of $\C$ is a zero object and each $2$-fold
    copower $A+A$ is a biproduct;
  \item Each $A \in \C$ bears a counital comagma structure, naturally
    in $A$.
  \end{enumerate}
\end{Prop}

\begin{proof}
  The implication $\text{(i)} \Rightarrow \text{(ii)}$ follows from
  Proposition~\ref{prop:1}; while
  $\text{(ii)} \Rightarrow \text{(iii)}$ follows since the diagonal
  morphism and the zero morphism are natural. For
  $\text{(iii)} \Rightarrow \text{(i)}$, let each $A \in \C$ bear
  counital comagma structure $(\delta_A, \varepsilon_A)$. For each
  $A,B \in \C$, the set $\C(A,B)$ bears a binary operation $m_{AB}$
  and constant $e_{AB}$ given by the respective composites:
  \begin{align*}
    \C(A,B)^2 \xrightarrow{\cong}
    \C(A+A,B) \xrightarrow{\C(\delta_A,B)} \C(A,B) \quad \text{and} \quad
    1 \xrightarrow{\cong}
    \C(0,B) \xrightarrow{\C(\varepsilon_A,B)} \C(A,B)\rlap{ ;}
  \end{align*}
  and by naturality of $\varepsilon$ and $\delta$, these operations
  are preserved by composition on each side. To deduce
  semi-additivity, it thus suffices to show that $(m_{AB}, e_{AB})$
  endows $\C(A,B)$ with commutative monoid structure. Since
  $(1_A+\varepsilon_A)\delta_A=1_A=(\varepsilon_A+1_A)\delta_A$, we
  see that $m_{AB}$ has $e_{AB}$ as a two-sided unit; it remains to
  prove associativity and commutativity. Naturality of $\delta$ at
  coproduct injections $\iota_i \colon A \rightarrow A + A$ (for
  $i = 1,2$) asserts the equalities
  $(\iota_i + \iota_i)\delta_A = (\delta_{A+A})\iota_i \colon A
  \rightarrow A + A + A + A$; moreover, the composite
  $(\iota_i + \iota_i)\delta_A$ is, by elementary
  properties of coproducts, equal to the composite
  \begin{equation*}
    A \xrightarrow{\iota_i} A + A \xrightarrow{\delta_A + \delta_A} A
    + A + A + A \xrightarrow{1 + \spn{\iota_2, \iota_1} + 1} A + A + A
    + A\rlap{ .}
  \end{equation*}
  Pairing these equalities together for $i=1,2$ thus yields
  commutativity in $\C$ 
  of each triangle on the left in:
  \begin{equation*}
    \cd[@C-1.1em]{
      & {A+A} \ar[dl]_-{\delta_{A} + \delta_A} \ar[dr]^-{\delta_{A + A}} & & &
      A \ar[rr]^-{\delta_A} \ar[d]_-{\delta_A} && {A+A} \ar[d]^-{\delta_{A+A}} \\
      {A + A + A + A} \ar[rr]_-{1 + \spn{\iota_2,\iota_1} + 1} & & 
      {A + A + A + A} & &
      A + A \ar[rr]_-{\delta_A + \delta_A} && A + A + A + A\rlap{ .}
    }
  \end{equation*}
  On the other hand, naturality of $\delta$ at $\delta_A$ says that
  each square on the right commutes. Combining these two and applying
  the hom-functor $\C(\thg, B)$, we conclude that $m_{AB}$ satifies
  the \emph{mediality} axiom
  $m_{AB}(m_{AB}(a,b),m_{AB}(c,d)) = m_{AB}(m_{AB}(a,c),m_{AB}(b,d))$.
  The Eckmann--Hilton argument~\cite{Eckmann1961Group-like} shows that
  any unital and medial binary operation on a set is associative and
  commutative, whence $(m_{AB}, e_{AB})$ endows each $\C(A,B)$ with a
  commutative monoid structure as required.
\end{proof}

We now prove our main theorems. First we recall:

\begin{Defn}
  \label{def:2}
  Let $(\C, \otimes, I)$ be a monoidal category. A \emph{right dual}
  for $X \in \C$ comprises an object $Y \in \C$ and a map
  $\eta \colon I \rightarrow Y \otimes X$ such that each
  $f \colon A \rightarrow B \otimes X$ is of the form
  \begin{equation*}
    A \xrightarrow{A \otimes \eta } A \otimes Y \otimes X \xrightarrow{g \otimes X} B \otimes X
  \end{equation*}
  for a unique $g \colon A \otimes Y \rightarrow B$. Note that $Y$ is
  right dual to $X$ if and only if there are natural isomorphisms
  $\theta_{A,B} \colon \C(A , B \otimes X) \rightarrow \C(A \otimes Y,
  B)$
  which are stable under tensor, in the sense that
  $(A' \otimes \thg) \circ \theta_{A,B} = \theta_{A' \otimes A, A'
    \otimes B} \circ (A' \otimes \thg) \colon \C(A,B \otimes X)
  \rightarrow \C(A' \otimes A \otimes Y, A' \otimes B)$.
\end{Defn}

\begin{proof}[Proof of Theorem~\ref{thm:3}.]
  If $\C$ is semi-additive, then by Proposition~\ref{prop:1}, the
  initial object is terminal and the coproduct
  $B\otimes(I+I) \cong B+B$ is a product. Thus there are natural
  bijections between morphisms $A \otimes 0 \rightarrow B$ and
  $A \rightarrow B \otimes 0$ on the one hand, and morphisms
  $A \otimes (I+I) \rightarrow B$ and $A \rightarrow B\otimes (I+I)$
  on the other; these isomorphisms are easily seen to be stable under
  tensor, whence both $0$ and $I+I$ are self-dual.

  In the converse direction, the assumption that $A \otimes (\thg)$
  preserves the $2$-fold copower $I+I$ implies that $\C$ has binary
  copowers, and so Proposition~\ref{prop:2} is applicable. Since there
  are natural isomorphisms $A \otimes (I+I) \cong A+A$ and
  $A \otimes 0\cong 0$, we may verify Proposition~\ref{prop:2}(iii) by
  constructing a counital comagma structure on the object $I \in \C$.

  To this end, let
  $\varepsilon \colon I \rightarrow Z \otimes 0 \cong 0$ exhibit $Z$
  as right dual to $0$ and
  $\eta \colon I \rightarrow D \otimes (I+I) $ exhibit $D$ as right
  dual to $I+I$. Lemma~\ref{lemma:1} applied to the natural family
  $\varepsilon_A=A \otimes \varepsilon \colon A \rightarrow 0$ shows
  that $0$ is terminal; in particular, $\varepsilon f = \varepsilon g$
  for any two maps $f, g \colon X \rightrightarrows I$. The defining
  property of the right dual $D$ now yields unique maps
  $\pi_1, \pi_2 \colon D \rightarrow I$ such that
  \begin{equation*}
    \iota_i = I \xrightarrow{\eta} D \otimes (I+I) \xrightarrow{\pi_i \otimes 1} I + I\rlap{ .}
  \end{equation*}
  Writing $\bar \eta$ for the composite
  $I \xrightarrow{\eta} D \otimes (I+I) \cong D + D$, it follows that
  \begin{equation*}
    \iota_i = I \xrightarrow{\bar \eta} D + D \xrightarrow{\pi_i + \pi_i} I + I\rlap{ .}
  \end{equation*}
  Now take $\delta \colon I \rightarrow I + I$ to be the composite
  $(\pi_1 + \pi_2) \bar \eta \colon I \rightarrow D + D \rightarrow I
  + I$
  and observe that
  $(1_I+\varepsilon)\delta = (1_I+\varepsilon)(\pi_1 + \pi_2)\bar\eta
  = (\pi_1 +\varepsilon \pi_2)\bar \eta = (\pi_1+\varepsilon
  \pi_1)\bar\eta = (1_I+\varepsilon)\iota_1 = 1_I$
  and dually $(\varepsilon+1_I)\delta = 1_I$. So $I$ bears counital
  comagma structure, whence each $A \in \C$ does so, naturally in $A$.
  It follows from Proposition~\ref{prop:2} that $\C$ is semi-additive.
\end{proof}

\begin{proof}[Proof of Theorem~\ref{thm:1}.]
  Propositions~\ref{prop:1} and \ref{prop:2} show that a category with
  finite coproducts is semi-additive if and only if it has finite
  biproducts. The claim therefore follows from Theorem~\ref{thm:3}.
\end{proof}

We conclude the paper with some examples showing that our
result is, in a certain sense, the best possible: the assumptions in
Theorem~\ref{thm:3} that $A \otimes (\thg)$ preserves the initial
object $0$ and the coproduct $I+I$ cannot be relaxed.

\begin{Ex}
  \label{ex:1}
  Let $\C$ be the category of endofunctors of the category
  $\mathbf{2}=\{0 < 1\}$, with the composition tensor product
  $F \otimes G \defeq F \circ G$.
  The category $\C$ is isomorphic to the ordered set
  $\{0 < \id < 1 \}$, so the coproduct $\id+\id$ is equal to $\id$ and
  is therefore both self-dual and also preserved by each
  $A \otimes (\thg)$. The initial object $0$ has as right dual the
  terminal object $1$, but $0$ is not preserved by $1 \otimes (\thg)$.
  The category $\C$ is not semi-additive since $0$ is not isomorphic
  to $1$.
\end{Ex}

\begin{Ex}
  \label{ex:2}
  Let $\cat{FinSet}_{\ast}$ be the category of finite pointed sets,
  and $\C$ the category of zero-object-preserving endofunctors of
  $\cat{FinSet}_\ast$, equipped with the composition tensor product. The
  constant functor at $0$ is self-dual, and $\id+\id$ has
  $\id \times \id$ as right dual (here: right adjoint). Every
  $A \otimes (\thg)$ preserves the initial object by assumption, but
  need not preserve the coproduct $\id + \id$. The category $\C$ is
  not semi-additive since the canonical morphism
  $\id + \id \rightarrow \id \times \id$ is not invertible.
\end{Ex}

Note that, in the ``if'' direction of our main results, we required
tensoring on the \emph{left} to preserves finite coproducts of $I$,
but for $0$ and $I+I$ to have \emph{right} duals. If we
reverse \emph{both} ``left'' and ``right'' here, then by duality we
again obtain sufficient conditions for semi-additivity, but the
preceding examples show that this may not be the case if we reverse only
\emph{one} of them. Indeed, in these examples, tensoring on the
\emph{right} preserves all finite coproducts, since these are computed
pointwise in functor categories, and $0$ and $I+I$ have \emph{right}
duals, but semi-additivity does not obtain. Dually, it may be the case
that tensoring on the \emph{left} preserves finite coproducts and that
$0$ and $I+I$ have \emph{left} duals without semi-additivity holding.

\end{document}